\documentclass[12pt]{article}
\usepackage{amsmath}
\usepackage{amsthm}
\usepackage{amsfonts}
\usepackage{amssymb}
\usepackage{esint}
\usepackage{eucal}
\usepackage{mathrsfs}
\usepackage{graphicx,graphics}
\numberwithin{equation}{section}
\usepackage{graphicx,graphics}
\usepackage{pdfsync}
\usepackage{multirow}
\usepackage{endnotes}
\usepackage{epstopdf}
\def \dis {\displaystyle}

\def \confai {-\kern -.5em\rightharpoonup}
\def \cqfd {\hfill$\Box$}
\def \div{\mbox{\rm div}}
\def \Div{\mbox{\rm Div}}
\def \curl{\mbox{\rm curl}}

\def \al {\alpha}
\def \be {\beta}
\def \ga {\gamma}

\def \De {\Delta}
\def \ep {\varepsilon}

\def \Om {\Omega}
\def \la {\lambda}
\def \La {\Lambda}
\def \ph {\varphi}
\def \th {\theta}

\def \Si {\Sigma}
\def \NN {\mathbb N}
\def \ZZ {\mathbb Z}

\def \RR {\mathbb R}

\def \D {\mathscr{D}}
\def \F {\mathscr{F}}

\def \beq {\begin{equation}}
\def \eeq {\end{equation}}
\def \ba {\begin{array}}
\def \ea {\end{array}}
\def \bs {\bigskip}
\def \ms {\medskip}

\def \ecart {\noalign{\medskip}}
\newtheorem{Thm}{Theorem}[section]

\newtheorem{Pro}[Thm]{Proposition}
\newtheorem{Lem}[Thm]{Lemma}

\newtheorem{Adef}[Thm]{Definition}
\newenvironment{Def}{\begin{Adef}\rm}{\end{Adef}}
\newtheorem{Arem}[Thm]{Remark}
\newenvironment{Rem}{\begin{Arem}\rm}{\end{Arem}}
\newtheorem{Aappl}[Thm]{Application}

\newtheorem{Aexa}[Thm]{Example}

\newtheorem{Anot}[Thm]{Notation}

\def \refe #1.{(\ref{#1})}
\def \reff #1.{figure~\ref{#1}}
\def \refs #1.{Section~\ref{#1}}
\def \refss #1.{Subsection~\ref{#1}}
\def \refD #1.{Definition~\ref{#1}}
\def \refT #1.{Theorem~\ref{#1}}
\def \refL #1.{Lemma~\ref{#1}}
\def \refC #1.{Corollary~\ref{#1}}
\def \refP #1.{Proposition~\ref{#1}}
\def \refPt #1.{Properties~\ref{#1}}
\def \refR #1.{Remark~\ref{#1}}
\def \refA #1.{Application~\ref{#1}}
\def \refE #1.{Example~\ref{#1}}
\def \refN #1.{Notation~\ref{#1}}
\newcounter{marnote}

\title{Isotropic realizability of a strain field for the incompressible two-dimensional Stokes equation}
\author{M. Briane\footnote{INSA de Rennes, IRMAR (CNRS, UMR 6625), FRANCE -- mbriane@insa-rennes.fr}}
\begin{document}
\maketitle
\begin{abstract}
In the paper we study the problem of the isotropic realizability in $\RR^2$ of a regular strain field $e(U)={1\over 2}\left(DU+DU^T\right)$ for the incompressible Stokes equation, namely the existence of a positive viscosity $\mu>0$ solving the Stokes equation in $\RR^2$ with the prescribed field~$e(U)$. We show that if $e(U)$ does not vanish at some point, then the isotropic realizability holds in the neighborhood of that point. The global realizability in $\RR^2$ or in the torus is much more delicate, since it involves the global existence of a regular solution to a semilinear wave equation the coefficients of which depend on the derivatives of $U$. Using the semilinear wave equation we prove a small perturbation result: If $DU$ is periodic and close enough to its average for the $C^4$-norm, then the strain field is isotropically realizable in a given disk centered at the origin. On the other hand, a counter-example shows that the global realizability in $\RR^2$ may hold without the realizability in the torus, and it is discussed in connection with the associated semilinear wave equation. The case where the strain field vanishes is illustrated by an example. The singular case of a rank-one laminate field is also investigated.
\end{abstract}
\par\bs\noindent
\section{Introduction}
In the theory of composites (see, {\em e.g.}, \cite{Mil}) the effective properties of a composite are classically obtained by the interactions of several isotropic phases periodically arranged, involving some periodic electric fields and current fields. It turns out that the related electric field may satisfies some constraints. Indeed, in two-dimensional conductivity Alessandrini and Nesi \cite{AlNe} showed the positivity of the determinant of the periodic matrix-valued electric field (each row of which corresponds to the vector electric field associated with one direction of the applied field). Hence, a two-dimensional matrix gradient field with a non-positive determinant cannot be an electric field. Therefore, it is natural to characterize the electric fields among all the possible gradient fields, and the current fields among all the possible divergence free fields. In this spirit a periodic gradient field $\nabla u$ is said to be isotropically realizable as an electric field in $\RR^d$ if there exists a positive conductivity $\sigma$ solving the equation
\beq\label{conequ}
\div\left(\sigma\nabla u\right)=0\quad\mbox{in }\RR^d.
\eeq
The isotropic realizability holds in the torus if moreover the conductivity can be chosen periodic. Following \cite{BMT} it is easy to build a periodic regular gradient field which is isotropically realizable in the whole space but not in the torus. In \cite{BMT} we have completely characterized the set of the periodic regular gradients as isotropically realizable electric fields using a gradient flow approach. So, in dimension two a periodic regular gradient field is shown to be isotropically realizable in $\RR^2$, if and only if it does not vanish in~$\RR^2$. Moreover, the isotropic realizability in the torus needs an extra assumption satisfied by the gradient flow. Similarly, a periodic divergence free field $j$ is said to be isotropically realizable if there exists a positive conductivity $\sigma$ such that $\sigma^{-1}j$ is a gradient. In \cite{BrMi} we have proved that in dimension three any periodic regular divergence free field is isotropically realizable under some geometrical assumptions. To this end, we have used a more sophistical approach based on three dynamical systems along the current field, its curl and the cross product with its curl. However, the characterization of the current fields is less complete than the characterization of the electric fields.
\par
In this paper we study the isotropic realizability of a strain field for the incompressible Stokes equation in dimension two. More precisely, let $U:\RR^2\to\RR^2$ be a regular divergence free field the gradient of which is possibly periodic. The question is the existence of a positive continuous viscosity $\mu:\RR^2\to(0,\infty)$ and a continuous pressure $p:\RR^2\to\RR$ such that the symmetrized gradient $e(U)={1\over 2}\left(DU+DU^T\right)$ is solution of the Stokes equation
\beq\label{Sequi}
-\,\Div\big(\mu\,e(U)\big)+\nabla p=0\quad\mbox{in }\RR^2.
\eeq
Since the strain field is a matrix, the dynamical system approach of \cite{BMT,BrMi} does not apply. Moreover, we have not succeeded to obtain a global realizability result in $\RR^2$ as general as for the electric fields and the current fields in \cite{BMT,BrMi}. The difficulty comes from the approach based on the existence of solutions to specific hyperbolic equations.
\par
In Section~\ref{s.locrea} we study the local realizability of a regular strain field. We prove (see Theorem~\ref{thm.isorea}) that if the strain field does not vanish at some point of $\RR^2$, then the isotropic realizability holds in a neighborhood of that point. Using in equation \eqref{Sequi} the representation in dimension two of a divergence free gradient field as an orthogonal gradient, we are led to a hyperbolic system which allows us to construct both a suitable viscosity $\mu$ and a pressure $p$.
\par
In Section~\ref{s.glorea} the question of the global realizability is investigated. To this end we consider the equivalent form of equation \eqref{Sequi}
\beq
\curl\left[\Div\big(\mu\,e(U)\big)\right]=0\quad\mbox{in }\RR^2,
\eeq
for which we search a positive solution of the type $\mu=e^u$. We are thus led to the semilinear wave equation
\beq\label{Sequi2}
e^{-u}\,\curl\left[\Div\big(e^u\,e(U)\big)\right]=0\quad\mbox{in }\RR^2,
\eeq
which must be satisfied by some regular function $u$ in $\RR^2$, the coefficients of which depend on the derivatives of the prescribed velocity $U$. It is well known (see, {\em e.g.}, \cite{Hor} and the references therein) that such a nonlinear wave equation does not admit necessarily a global regular solution for given initial data. However, it is not clear that one cannot choose some suitable initial data which induce a global solution. This is the crucial point related to the question of the global realizability.
Due to this difficulty we have not obtained a global realizability result but only a quasi-global realizability result under a small perturbation condition.
More explicitly, we have proved the following result (see Theorem~\ref{thm.isoreaper}): Let $M$ be a matrix in $\RR^{2\times 2}$ with zero trace and $M+M^T\neq 0$, and let $U:\RR^2\to\RR^2$ be a regular divergence free field such that $DU$ is periodic with average~$M$. For any $R>0$, if the norm of $DU$ in $C^4(\RR^2)^{2\times 2}$ is less than some value $\ep_R>0$, then the strain field $e(U)$ is isotropically realizable in the disk $D(0,R)$ centered at the origin and of radius $R$. Unhappily, it is difficult to estimate the value $\ep_R$ with respect to~$R$, since it is linked to the lifespan of the solutions $u$ to the semilinear wave equation \eqref{Sequi2}, which is not known as above mentioned. So, $\ep_R$ could tend to $0$ as $R\to\infty$, which would prevent the global isotropic realizability of the strain $e(U)$ in $\RR^2$. When the realizability of the strain field is not realizable in the torus,  the following alternative holds (see Proposition~\ref{pro.alt}):
\begin{enumerate}
\item the semilinear wave equation \eqref{Sequi2} has not a global regular solution $u$,
\item any global regular solution $u$ to \eqref{Sequi2} is either not bounded or not uniformly continuous in~$\RR^2$.
\end{enumerate}
\par
Section~\ref{s.exa} is devoted to singular cases. Section~\ref{ss.cex} deals with the strain field $e(U_\ep)$, $\ep>0$, defined by
\beq
e(U_\ep):=\begin{pmatrix} 1 & \ep\sin(2\pi y) \\ \ep\sin(2\pi y) & -1 \end{pmatrix}\quad\mbox{for }(x,y)\in\RR^2,
\eeq
which is shown to be not isotropically realizable in the torus (see Proposition~\ref{pro.norea}). The main result of Section~\ref{s.glorea} thus implies that $e(U_\ep)$ is isotropically realizable in the given disk $D(0,R)$ provided that $\ep$ is small enough. The strain field $e(U_\ep)$ is actually isotropically realizable in~$\RR^2$ with the viscosity $\mu(x,y):=e^{2\pi x}$. Hence, by virtue of the above alternative the regular solutions~$u$ (including the particular solution $u(x,y):=2\pi x$) of the equation
\beq
\curl\left[\Div\big(e^u\,e(U_\ep)\big)\right]=0\quad\mbox{in }\RR^2,
\eeq
are either not bounded or not uniformly continuous in $\RR^2$.
\par
Next, Section~\ref{ss.exavan} is devoted to a case where the strain field vanishes at some point. In the electric field framework the Hartman-Wintner theorem (see \cite{HaWi}, \cite{Sch} Chap.~7) claims that if $\nabla u$ is a realizable (namely $u$ is solution to some equation \refe{conequ}.) regular non-zero electric field in $\RR^2$ with $\nabla u(X^*)=0$, then the critical point $X^*$ is isolated and the following condition holds (see~\cite{Ale}, Remark 1.2):
\beq\label{HWcon}
\exists\,n\in\NN\setminus\{0\},\ \exists\,C\geq 1,\quad C^{-1}\,|X-X^*|^n\leq|\nabla u(X)|\leq C\,|X-X^*|^n
\quad\mbox{for $X$ close to $X^*$}.
\eeq
Up to our knowledge there is no similar result for a strain field in $\RR^2$ which vanishes at some point.
We have simply proved (see Proposition~\ref{pro.sepvar}) that the particular regular strain field only vanishing at the origin:
\beq
e(U):=\begin{pmatrix} 0 & f(x)+g(y) \\ f(x)+g(y) & 0 \end{pmatrix}
\quad\mbox{where } f,g\in C^\infty\big([-1,1]^2\big)\mbox{ and }f(0)=g(0)=0,
\eeq
is isotropically realizable in the neighborhood of the origin, if and only if
\beq
\exists\,a>0,\quad e(U)(X)=a\,|X|^2\begin{pmatrix} 0 & 1 \\ 1 & 0 \end{pmatrix}+o\big(|X|^2\big),
\eeq
which is sharper than the Hartman-Wintner condition~\eqref{HWcon}.
\par
Finally, in Section~\ref{ss.lam} we study a rank-one laminate strain field which takes only two values and is thus not continuous. We give (see Theorem~\ref{thm.realam}) a necessary and sufficient condition on the two phases so that the rank-one laminate strain field is isotropically realizable with a similar rank-one laminate viscosity.
\subsection*{Notations}
\begin{itemize}
\item $I_2$ denotes the unit matrix of $\RR^2$, and $R_\perp:=\begin{pmatrix} 0 & -1 \\ 1 & 0 \end{pmatrix}$.
\item For $A\in\RR^{2\times 2}$, $A^T$ denotes the transposed of the matrix $A$.
\item $\cdot$ denotes the scalar product in $\RR^d$.
\item For any matrices $A,B\in\RR^{2\times 2}$, $A:B:={\rm tr}\left(A^T\!B\right)$, where tr denotes the trace of a matrix, and $|A|:=\sqrt{{\rm tr}\left(A^T\!A\right)}$ is the Frobenius norm in $\RR^{2\times 2}$.
\item $\RR^{2\times 2}_{s,0}$ denotes the set of the symmetric matrices of $\RR^{2\times 2}$ with zero trace.
\item For $\xi,\eta\in\RR^2$, $\dis \xi\otimes\eta:=\big[\xi_i\,\eta_j\big]_{1\leq i,j\leq 2}$ and
$\dis \xi\odot\eta:={1\over 2}\left(\xi\otimes\eta+\eta\otimes\xi\right)$.
\item For $u\in C^1(\RR^2)$, the partial derivatives $\dis {\partial u\over\partial x}$ and $\dis {\partial u\over\partial y}$ are respectively denoted $\partial_x u$ and $\partial_y u$. Moreover, the gradient of $u$ is denoted $\dis \nabla u=\begin{pmatrix}\partial_x u \\ \partial_y u \end{pmatrix}$.
\item For $u\in C^2(\RR^2)$, the Hessian matrix of $u$ is denoted
$\nabla^2 u:=\begin{pmatrix}\partial^2_{xx}u & \partial^2_{xy}u \\ \partial^2_{yx}u & \partial^2_{yy}u \end{pmatrix}.$ 
\item For $U=\begin{pmatrix}U_x \\ U_y\end{pmatrix}\in C^1(\RR^2)^2$, the curl of $U$ is $\curl\,U:=\partial_x U_y-\partial_y U_x$, the gradient of $U$ is $DU:=\begin{pmatrix}\partial_x U_x & \partial_x U_y \\ \partial_y U_x & \partial_y U_y \end{pmatrix}$, 
and the strain tensor $e(U)$ is defined by
\beq\label{e(U)}
e(U):={1\over 2}\left(DU+DU^T\right)=\begin{pmatrix}\partial_x U_x & {1\over 2}\left(\partial_x U_y+\partial_y U_x\right)
\\
{1\over 2}\left(\partial_x U_y+\partial_y U_x\right) & \partial_y U_y
\end{pmatrix}.
\eeq
\item For $\Si=\begin{pmatrix}\Si_{xx} & \Si_{xy} \\ \Si_{yx} & \Si_{yy}\end{pmatrix}\in C^1(\RR^d)^{2\times 2}$, the divergence of $\Si$ is
$\Div\left(\Si\right):=\begin{pmatrix}\partial_x\Si_{xx}+\partial_y\Si_{yx} \\ \partial_x\Si_{xy}+\partial_y\Si_{yy}\end{pmatrix}$.
\end{itemize}
%%%%%%%%%%
\section{Local realizability for a non-vanishing field}\label{s.locrea}
Let $U:\RR^2\to\RR^2$, $U=(U_x,U_y)$, be a regular divergence free field, and let $X_*\in\RR^2$. The question is to know if the strain field $e(U)$ is {\em isotropically realizable} for the Stokes equation in the neighborhood of the point $X_*$. More precisely, does there exist a neighborhood $\Om$ of $X_*$, a positive continuous viscosity $\mu$ in $\Om$ and a continuous pressure $p$ in $\Om$ such that
\beq\label{Stokes}
-\,\Div\,\big(\mu\,e(U)\big)+\nabla p=0\quad\mbox{in }\Om?
\eeq
The following result provides a sufficient condition of local realizability.
\begin{Thm}\label{thm.isorea}
Let $U$ be a divergence free field in $C^4(\RR^2)^2$, and let $X_*\in\RR^2$ be such that
\beq\label{e(U)­0}
e(U)(X_*)\neq 0.
\eeq
Then, there exist an open neighborhood $\Om$ of $X_*$, a positive function $\mu$ in $C^0(\Om)$ and a function $p$ in $C^0(\Om)$, such that the vector field $U$ solves the Stokes equation \eqref{Stokes} in $\Om$.
\end{Thm}
\noindent
{\bf Proof of Theorem~\ref{thm.isorea}.} By a translation we are led to $X_*=(0,0)$. Let $\Om$ be an open disk of positive radius centered on $(0,0)$. Since $U$ is divergence free, it can be written
\beq\label{Uu}
U=R_{\perp}\nabla u= \begin{pmatrix}-\,\partial_y u \\ \partial_x u \end{pmatrix},\quad\mbox{where }u\in C^5(\bar\Om).
\eeq
Due to the condition \eqref{e(U)­0} combined with the regularity of $U$, we may choose $\Om$ such that
\beq\label{e(U)­0a}
\partial^2_{xy}u\neq 0\;\;\mbox{in }\bar\Om\quad\mbox{or}\quad\partial^2_{xx}u-\partial^2_{yy}u\neq 0\;\;\mbox{in }\bar\Om.
\eeq
Moreover, the change of variables $u'(x',y')=u(x'+y',x'-y')$ yields
\beq
\partial^2_{xy}u={1\over 2}\left(\partial^2_{x'x'}u'-\partial^2_{y'y'}u'\right).
\eeq
Hence, the first condition of \eqref{e(U)­0a} leads us to the second one. Therefore, from now one we assume that (up to change $u$ in $-u$)
\beq\label{e(U)­0b}
\partial^2_{xx}u-\partial^2_{yy}u>0\quad\mbox{in }\bar\Om.
\eeq
\par
The proof is now divided in three steps. In the first step, from two suitable solutions $v^+,v^-$ of second-order hyperbolic equations we build a continuous viscosity $\mu$ in the neighborhood of the point $(0,0)$. In the second step, we prove the existence of a function $v^+$ for $x>0$. The third step is devoted to the existence of a function $v^-$ for $x<0$.
\par\bs\noindent
{\it First step: Construction of an admissible viscosity $\mu$.}
\par\noindent
Assume for the moment that there exist a number $\tau>0$ with $[-\tau,\tau]^2\subset\Om$, and two functions $v^+\in C^2\big([0,\tau]\times[-\tau,\tau]\big)$, $v^-\in C^2\big([-\tau,0]\times[-\tau,\tau]\big)$ satisfying
\beq\label{v+-}
\left\{\ba{lll}
\partial^2_{xy}v^+>0 & \mbox{in }[0,\tau]\times[-\tau,\tau]
\\ \ecart
\partial^2_{xy}v^->0 & \mbox{in }[-\tau,0]\times[-\tau,\tau]
\\ \ecart
\partial^2_{xy}v^+(0,\cdot)=\partial^2_{xy}v^-(0,\cdot) & \mbox{in }[-\tau,\tau]
\\ \ecart
\partial^2_{yy}v^+(0,\cdot)=\partial^2_{yy}v^-(0,\cdot) & \mbox{in }[-\tau,\tau],
\ea\right.
\eeq
and such that the divergence free functions $V^+:=R_{\perp}\nabla v^+$ and $V^-:=R_{\perp}\nabla v^-$ satisfy
\beq\label{V+-}
e(U):e(V^+)=0\;\;\mbox{in }[0,\tau]\times[-\tau,\tau]\quad\mbox{and}\quad e(U):e(V^-)=0\;\;\mbox{in }[-\tau,0]\times[-\tau,\tau].
\eeq
Then, define the function~$\mu$ by
\beq\label{muv+-}
\mu:=\left\{\ba{ll}
\dis -\,{2\,\partial_{x}V^+_x\over\partial_{x}U_y+\partial_{y}U_x}={2\,\partial^2_{xy}v^+\over\partial^2_{xx}u-\partial^2_{yy}u}
& \mbox{in }[0,\tau]\times[-\tau,\tau]
\\ \ecart
\dis -\,{2\,\partial_{x}V^-_x\over\partial_{x}U_y+\partial_{y}U_x}={2\,\partial^2_{xy}v^-\over\partial^2_{xx}u-\partial^2_{yy}u}
& \mbox{in }[-\tau,0)\times[-\tau,\tau],
\ea\right.
\eeq
which is continuous and positive in $[-\tau,\tau]^2$ by virtue of \eqref{e(U)­0b} and \eqref{v+-}.
Also define the function~$p$ by
\beq\label{pv+-}
p:=\left\{\ba{ll}
\mu\,\partial_{x}U_x-\partial_{y}V^+_x=-\,\mu\,\partial^2_{xy}u+\partial^2_{yy}v^+
& \mbox{in }[0,\tau]\times[-\tau,\tau]
\\ \ecart
\mu\,\partial_{x}U_x-\partial_{y}V^-_x=-\,\mu\,\partial^2_{xy}u+\partial^2_{yy}v^-
& \mbox{in }[-\tau,0)\times[-\tau,\tau],
\ea\right.
\eeq
which is continuous in $[-\tau,\tau]^2$ by the fourth condition of \eqref{v+-}.
\par
By the free divergence of $U$, \eqref{V+-} and the definition \eqref{muv+-} of $\mu$, we have
\beq\label{pUV}
\ba{ll}
\dis p-\mu\,\partial_{y}U_y-\partial_{x}V^\pm_y & =\mu\,\partial_{x}U_x-\partial_{y}V^\pm_x-\mu\,\partial_{y}U_y-\partial_{x}V^\pm_y
\\ \ecart
& \dis =-\,{4\,\partial_{x}U_x\,\partial_{x}V^\pm_x\over\partial_{x}U_y+\partial_{y}U_x}-\partial_{y}V^\pm_x-\partial_{x}V^\pm_y
\\ \ecart
& \dis =-\,{2\,e(U):e(V^\pm)\over\partial_{x}U_y+\partial_{y}U_x}=0.
\dis 
\ea
\eeq
Hence, from \eqref{muv+-}, \eqref{pv+-} and \eqref{pUV} we deduce that $U$ and $V$ are solutions of the system
\beq\label{sysUV}
\left\{\ba{ll}
-\,\mu\,\partial_{x}U_x+p & =-\,\partial_{y}V^\pm_x
\\ \ecart
\dis -\,{\mu\over 2}\left(\partial_{x}U_y+\partial_{y}U_x\right) & =\partial_{x}V^\pm_x
\\ \ecart
\dis -\,{\mu\over 2}\left(\partial_{x}U_y+\partial_{y}U_x\right) & =-\,\partial_{y}V^\pm_y
\\ \ecart
-\,\mu\,\partial_{y}U_y+p & =\partial_{x}V^\pm_y
\ea\right.
\quad\mbox{in }\big([0,\tau]\times[-\tau,\tau]\big)\cup\big([-\tau,0]\times[-\tau,\tau]\big),
\eeq
which is equivalent to
\beq
-\,\mu\,e(U)+p\,I_2=R_{\perp}DV^\pm\quad\mbox{in }\big([0,\tau]\times[-\tau,\tau]\big)\cup\big([-\tau,0]\times[-\tau,\tau]\big).
\eeq
Therefore, we get that
\beq\label{Sto}
-\,\Div\,\big(\mu\,e(U)\big)+\nabla p=\Div\left(R_{\perp}DV^\pm\right)=0
\quad\mbox{in }\big((0,\tau)\times(-\tau,\tau)\big)\cup\big((-\tau,0)\times(-\tau,\tau)\big).
\eeq
This combined with the continuity of the strain tensor $\mu\,e(U)$ and the pressure $p$ at the interface $\{0\}\times[-\tau,\tau]$ implies that $U$ is solution of the Stokes equation \eqref{Stokes} replacing $\Om$ by $(-\tau,\tau)^2$.
\par\bs\noindent
{\it Second step: Existence of a function $v^+$ for $x\geq 0$.}
\par\noindent
Recall that $\Om$ is a regular simply connected neighborhood of $(0,0)$.
Let $a$ be the function defined by
\beq\label{au}
a:=-\,{2\,\partial_x U_x\over \partial_x U_y+\partial_y U_x}={2\,\partial^2_{xy}u\over\partial^2_{xx}u-\partial^2_{yy}u}
\quad\mbox{in }\bar\Om,
\eeq
and let $A$ be the matrix-valued function defined by
\beq\label{Au}
A:=\begin{pmatrix} \al & 0 \\ \ga & \be \end{pmatrix}\quad\mbox{in }\Om,\quad\mbox{where}\quad
\left\{\ba{l}\al:=a-\sqrt{a^2+1} \\ \ecart \be:=a+\sqrt{a^2+1} \\ \ecart \ga:=-\,\partial_x\al-\be\,\partial_y\al.\ea\right.
\eeq
Consider the semilinear hyperbolic system given for $V=\begin{pmatrix} v \\ w \end{pmatrix}$ by
\beq\label{semhyp}
\partial_x V+A\,\partial_y V=\begin{pmatrix} \partial_x v+\al\,\partial_y v \\ \partial_x w+\be\,\partial_y w+\ga\,\partial_y v  \end{pmatrix}
=\begin{pmatrix} w \\ 0 \end{pmatrix}.
\eeq
System \eqref{semhyp} is strictly hyperbolic since the eigenvalues of $A$ satisfy $\al<\be$. As $u$ is in $C^5(\bar\Om)$, the matrix-valued function $A$ belongs to $C^2(\bar\Om)^{2\times 2}$. Moreover, we can extend the function $a$ in $\RR^2\setminus\bar\Om$ to a function in $C^3_b(\RR^2)$ ({\em i.e.} all the derivatives until order $3$ of the function $a$ are bounded in $\RR^2$), still denoted by $a$. Similarly, $A$ can be extended to a function in $C^2_b(\RR^2)^{2\times 2}$, still denoted by $A$. Let $c>0$ be a constant such that
\beq\label{abc}
|\al|+|\be|\leq c\quad\mbox{in }\RR^2,
\eeq
As a consequence, the characteristics associated with system \eqref{semhyp}, $Y(t;x,y)$ and $Z(t;x,y)$ solutions of the ordinary differential equations
\beq\label{YZ}
\left\{\ba{ll}
\dis {\partial Y\over\partial t}(t;x,y)=\al\big(t,Y(t;x,y)\big)
\\ \ecart
Y(x;x,y)=y,
\ea\right.
\left\{\ba{ll}
\dis {\partial Z\over\partial t}(t;x,y)=\be\big(t,Y(t;x,y)\big)
\\ \ecart
Z(x;x,y)=y,
\ea\right.
\quad\mbox{for }t\in\RR,
\eeq
define two functions in $C^3(\RR^3)$ (see, {\em e.g.}, \cite{HSD} Chapter~17).
Let $\D_c$ be the domain defined by
\beq\label{rDc}
\D_c:=\big\{(x,y)\in\RR^2\,:\,x\geq 0,\ -1+cx\leq y\leq 1-cx\big\},
\eeq
and let $v_0$, $w_0$ be two prescribed functions in $C^2\big([-1,1]^2\big)$. Then, by the Theorems~3.1 and~3.6 of \cite{Bre} there exists a unique solution $V$ in $C^2(\D_c)^2$, defined along the characteristics, of the hyperbolic system \eqref{semhyp} with the initial condition
\beq\label{V(0)}
V(0,y)=\begin{pmatrix}v_0(y) \\ w_0(y)\end{pmatrix}\quad\mbox{for }y\in [-1,1].
\eeq
Since by \eqref{semhyp} and \eqref{YZ}
\beq
{d\over dt}\left[v\big(t,Y(t;x,y)\big)\right]=w\big(t,Z(t;x,y)\big),\quad
{d\over dt}\left[w\big(t,Z(t;x,y)\big)\right]=-\,(\ga\,\partial_y v)\big(t,Z(t;x,y)\big),
\eeq
we get, taking into account \eqref{V(0)} and choosing $t=x$, the following integral representation of the solution $V$,
\beq\label{vw}
\left\{\ba{ll}
\dis v(x,y)=v_0\big(Y(0;x,y)\big)+\int_0^x w\big(s,Y(s;x,y)\big)\,ds
\\ \ecart
\dis w(x,y)=w_0\big(Z(0;x,y)\big)-\int_0^x (\ga\,\partial_y v)\big(s,Z(s;x,y)\big)\,ds,
\ea\right.
\quad\mbox{for any }(x,y)\in\D_c.
\eeq
Moreover, we have
\beq
\left\{\ba{ll}
\dis {\partial(\partial_x Y)\over\partial t}(t;x,y)=\partial_y\al\big(t,Y(t;x,y)\big)\,\partial_x Y(t;x,y)
\\ \ecart
\dis \al(x,y)+\partial_x Y(x;x,y)={\partial\over\partial x}\big(Y(x;x,y)\big)=0,
\ea\right.
\quad\mbox{for }t\in\RR,
\eeq
and similarly
\beq
\left\{\ba{ll}
\dis {\partial(\partial_y Y)\over\partial t}(t;x,y)=\partial_y\al\big(t,Y(t;x,y)\big)\,\partial_y Y(t;x,y)
\\ \ecart
\dis \partial_y Y(x;x,y)=1,
\ea\right.
\quad\mbox{for }t\in\RR.
\eeq
Hence, it follows that
\beq\label{Yxy}
\left\{\ba{l}
\dis \partial_x Y(t;x,y)=-\,\al(x,y)\,\exp\left(\int_x^t\partial_y\al\big(s;Y(s;x,y)\big)\,ds\right)
\\ \ecart
\dis \partial_y Y(t;x,y)=\exp\left(\int_x^t\partial_y\al\big(s;Y(s;x,y)\big)\,ds\right).
\ea\right.
\quad\mbox{for }t\in\RR.
\eeq
\par
On the other hand, consider any constant $\tau>0$ such that
\beq\label{taurDc}
[0,\tau]\times[-\tau,\tau]\subset\D_c.
\eeq
Also denote the function $v$ by $v^+$.
Then, \eqref{vw} combined with \eqref{Yxy} yields
\beq\label{dv+(0,y)}
\left\{\ba{l}
\partial_x v^+(0,y)=w_0(y)-\al(0,y)\,v_0'(y)
\\ \ecart
\partial_y v^+(0,y)=v_0'(y),
\ea\right.
\quad\mbox{for }y\in [-\tau,\tau].
\eeq
Moreover, the hyperbolic system \eqref{semhyp} combined with \eqref{Au}  implies that
\beq
\ba{ll}
0 & =\partial_x\left(\partial_x v^++\al\,\partial_y v^+\right)+\be\,\partial_y\left(\partial_x v^++\al\,\partial_y v^+\right)+\ga\,\partial_y v^+
\\ \ecart
& =\partial^2_{xx}v^++\al\be\,\partial^2_{yy}v^++\left(\al+\be\right)\,\partial^2_{xy}v^+,
\ea
\eeq
which yields the equation
\beq\label{equv+}
\partial^2_{xx}v^+-\partial^2_{yy}v^++2a\,\partial^2_{xy}v^+=0\quad\mbox{in }[0,\tau]\times[-\tau,\tau].
\eeq
Finally, defining the function $V^+:=R_{\perp}\nabla v^+$ and using the definition of $a$ in \eqref{au}, we deduce from \eqref{equv+} that
\beq\label{equV+}
\ba{ll}
e(U):e(V^+)
& \dis  = 2\,\partial_x U_x\,\partial_x V^+_x+{1\over 2}\left(\partial_x U_y+\partial_y U_x\right)\left(\partial_{x}V^+_y+\partial_{y}V^+_x\right)
\\ \ecart
& \dis = {1\over 2}\left(\partial_x U_y+\partial_y U_x\right)\left(-\,2a\,\partial_x V^+_x+\partial_{x}V^+_y+\partial_{y}V^+_x\right)
\\ \ecart
& \dis = {1\over 2}\left(\partial_x U_y+\partial_y U_x\right)\left(2a\,\partial^2_{xy}v^++\partial^2_{xx}v^+-\partial^2_{yy}v^+\right)=0
\;\;\mbox{in }[0,\tau]\times[-\tau,\tau],
\ea
\eeq
which corresponds to the first equation of \eqref{V+-}.
\par\bs\noindent
{\it Third step: Existence of a function $v^-$ for $x\leq 0$.}
\par\noindent
Define $\tilde{a}(x,y):=-\,a(-x,y)$ for $x\geq 0$ and $y\in\RR$, and let $\tilde{\al}$, $\tilde{\be}$ be the functions defined from $\tilde{a}$  as in formula \eqref{Au}. We can also assume that $\tilde{\al}$ and $\tilde{\be}$ satisfy the bound \eqref{abc} with some constant $c>0$. Then, following the approach of the second step, we get that for any $\tilde{v}_0$, $\tilde{w}_0$ in $C^2\big([-1,1]\big)$, there exists a function $\tilde{v}\in C^2(\D_c)$ satisfying for some $\tau>0$ of \eqref{taurDc},
\beq\label{dtv(0,y)}
\left\{\ba{l}
\partial_x\tilde{v}(0,y)=\tilde{w}_0(y)-\tilde{\al}(0,y)\,\tilde{v}_0'(y)
\\ \ecart
\partial_y\tilde{v}(0,y)=\tilde{v}_0'(y),
\ea\right.
\quad\mbox{for }y\in [-\tau,\tau],
\eeq
and
\beq\label{equtv}
\partial^2_{xx}\tilde{v}-\partial^2_{yy}\tilde{v}+2\tilde{a}\,\partial^2_{xy}\tilde{v}=0\quad\mbox{in }[0,\tau]\times[-\tau,\tau].
\eeq
Defining $v^-(x,y):=\tilde{v}(-x,y)$, for $(x,y)\in [-\tau,0]\times[-\tau,\tau]$, we thus deduce from \eqref{dtv(0,y)} and \eqref{equtv} that
\beq\label{dv-(0,y)}
\left\{\ba{l}
\partial_x v^-(0,y)=-\,\tilde{w}_0(y)+\tilde{\al}(0,y)\,\tilde{v}_0'(y)
\\ \ecart
\partial_y v^-(0,y)=\tilde{v}_0'(y),
\ea\right.
\quad\mbox{for }y\in [-\tau,\tau],
\eeq
and
\beq\label{equv-}
\partial^2_{xx} v^--\partial^2_{yy} v^-+2a\,\partial^2_{xy} v^-=0\quad\mbox{in }[-\tau,0]\times[-\tau,\tau].
\eeq
Moreover, as for \eqref{equv+} the equation \eqref{equv-} is equivalent to the second equation of \eqref{V+-} with $V^-:=R_{\perp}\nabla v^-$.
\par
Finally, choose $v_0(y)=\tilde{v}_0(y)=0$, $w_0(y)=y$ and $\tilde{w}_0(y)=-y$, for $y\in[-\tau,\tau]$. On the one hand, by \eqref{dv+(0,y)} and \eqref{dv-(0,y)} we have
\beq
\partial^2_{xy} v^+(0,y)=\partial^2_{xy} v^-(0,y)=1,\quad\partial^2_{yy} v^+(0,y)=\partial^2_{yy} v^-(0,y)=0,
\quad\mbox{for }y\in[-\tau,\tau],
\eeq
so that the two last equations of \eqref{v+-} are satisfied. On the other hand, since
\[
\partial^2_{xy} v^+(0,0)=\partial^2_{xy} v^-(0,0)=1,
\]
we can take $\tau>0$ small enough such that the two inequalities of \eqref{v+-} also hold. Therefore, the proof of Theorem~\ref{thm.isorea} is complete. \cqfd
%%%%%%%%%%
\section{Global realizability under a small perturbation}\label{s.glorea}
\subsection{The main result}
Contrary to \cite{BrMi,BMT} for fields in Electrostatics, we have not succeeded for the moment to prove a realizability result for strain fields in the whole space or in the torus. The difficulty is strongly linked to the derivation of global regular solutions to nonlinear wave equations. However, we have obtained a nearly global perturbation result for any periodic regular strain field sufficiently close to its average:
\begin{Thm}\label{thm.isoreaper}
Let $M$ be a matrix in $\RR^{2\times 2}$ with zero trace and $M+M^T\neq 0$.
Let $U$ be a divergence free function in $C^5(\RR^2)^2$ such that $X\mapsto U(X)-MX$ is $Y$-periodic.
For any $R>0$, there exists $\ep>0$ only depending on $R$ such that if
\beq\label{UMep}
\big\|U(X)-MX\big\|_{C^5_\sharp(Y)^{2\times 2}}<\ep,
\eeq
then the field $e(U)$ is isotropically realizable in the open disk $D(0,R)$ with a positive viscosity $\mu_R\in C^2\big(D(0,R)\big)$.
We can construct an admissible viscosity $\mu_R$ so that, for any $S>R>0$, if \eqref{UMep} holds for $R$ and $S$, then $\mu_S$ agrees with $\mu_R$ in $D(0,R)$.
Moreover, if $\ep$ is bounded from below by a positive constant independent of $R$, the strain field $e(U)$ is isotropically realizable in the whole space $\RR^2$ with a positive viscosity $\mu\in C^2(\RR^2)$.
% If $\ln\mu\in L^\infty(\RR)$ then $e(U)$ is isotropically realizable in the torus with a periodic positive viscosity $\mu_\sharp\in L^\infty_\sharp(Y)$.
\end{Thm}
The proof of Theorem~\eqref{thm.isoreaper} is based on the following perturbation result communicated by P. G\'erard \cite{Ger}, which already explains the limitation of the realizability result:
\begin{Lem}\label{lem.G}
Let $H(t,z,\la)$ be a polynomial of degree two in $\la\in\RR^2$,
\[
H(t,z,\la)=B(t,z)\la\cdot\la+V(t,z)\cdot\la+h(t,z)\quad\mbox{for }(t,z)\in\RR^2,\ \la\in\RR^2,
\]
the coefficients of which are in $C^2(\RR^2)$ in the variables $(t,z)\in\RR^2$.
Assume that
\beq\label{H1}
\La:=\int_\RR\left(\|B(t,\cdot)\|_{L^\infty(\RR)}+\|V(t,\cdot)\|_{L^\infty(\RR)}\right)dt<\infty.
\eeq
Then, there exists a positive constant $\ep$ depending on $\La$ such that, if
\beq\label{H2}
\int_{\RR}\|h(t,\cdot)\|_{L^\infty(\RR)}\,dt<\varepsilon,
\eeq
the semilinear wave equation
\beq\label{slwaeqG}
\left\{\ba{ll}
\square w(t,z)=\partial^2_{tt}w(t,z)-\partial^2_{zz}w(t,z)=H\big(t,z,\nabla w(t,z)\big), & (t,z)\in\RR^2
\\ \ecart
w(0,z)=\partial_t w(0,z)=0, & z\in\RR,
\ea\right.
\eeq
has a unique global solution $w$ in $C^1(\RR^2)$.
\end{Lem}
\begin{Rem}
The proof of the existence is an adaptation of the energy integral method that can be found for instance in H\"ormander's book \cite{Hor}, Chap.~VI. Moreover, the uniqueness follows from Theorem~6.4.10 of \cite{Hor}.
On the other hand, the smallness condition \eqref{H2} is essential to control the lifespan of the solution in the Gronwall type estimates.
Indeed, P. G\'erard has also provided in \cite{Ger} an example of equation~\eqref{slwaeqG} whose coefficients have compact support (which thus implies condition~\eqref{H1}), and which leads to a blow-up in finite time when the coefficient $h$ is not small enough.
\end{Rem}
\noindent
{\bf Proof of Theorem~\ref{thm.isoreaper}.}
Let $M$ be a matrix in $\RR^{2\times 2}$ with zero trace and $M+M^T\neq 0$.
Let $U$ be a divergence free function in $C^5(\RR^2)^2$ such that $X\mapsto U(X)-MX$ is $Y$-periodic.
Eliminating the pressure term in \eqref{Stokes} the strain field $e(U)$ is isotropically realizable in a simply connected domain $\Om$ of $\RR^2$ with a positive viscosity $\mu\in C^2(\Om)$, if and only if
\beq\label{Stokes2}
\curl\left[\Div\,\big(\mu\,e(U)\big)\right]=0\quad\mbox{in }\Om.
\eeq
The natural idea is to search a suitable positive viscosity of type $\mu=e^u$. Then, a lengthy but easy computation shows that $\mu$ satisfies the Stokes equation \eqref{Stokes2}, if and only if $u$ is solution of the semilinear wave equation
\beq\label{slwequ}
A:\nabla^2 u=-A\nabla u\cdot\nabla u+R_\perp\De U\cdot\nabla u-{1\over 2}\,\De(\curl\,U)\quad\mbox{in }\Om,
\eeq
where
\beq\label{A}
A:=e(U)R_\perp=\begin{pmatrix}
{1\over 2}\left(\partial_x U_y+\partial_y U_x\right) & -\,\partial_x U_x
\\
-\,\partial_x U_x & -\,{1\over 2}\left(\partial_x U_y+\partial_y U_x\right)
\end{pmatrix}
=\begin{pmatrix} a & b \\ b & -\,a\end{pmatrix}.
\eeq
The proof is now divided in three steps. In the first step we globally transform the wave equation \eqref{slwequ} into a canonical form.  In the second step we truncate some of the coefficients in the modified equation \eqref{slwequ} by coefficients with compact support in view of applying Lemma~\ref{lem.G}. In the third step we use Lemma~\ref{lem.G} to conclude.
\par\ms\noindent
{\it First step: Global transformation into a canonical form.}
\par\noindent
Due to $M+M^T\neq 0$, up to make the linear change of variables $(x,y)\mapsto(x+y,x-y)$ which permits to commute the entries $a$ and $b$ of the matrix-valued function $A$ \eqref{A}, we can assume that $M_{12}+M_{21}\neq 0$, so that for $\ep$ small enough, estimate \eqref{UMep} holds with
\beq\label{ane0}
\forall\,(x,y)\in\RR^2,\quad |a(x,y)|>{1\over 2}\,|M_{12}+M_{21}|-\ep>0.
\eeq
The local transformation of the wave equation \eqref{slwequ} into a canonical form is classical (see, {\em e.g.}, \cite{Eva}, Section 7.2). The global transformation is perhaps less classical and is based on the following change of variables.
\begin{Lem}\label{lem.canonic}
Consider a matrix $M$ and a vector-valued function $U$ satisfying the assumptions of Theorem~\ref{thm.isoreaper}. Let $R(x,\xi)$ and $S(x,\eta)$, for $\xi,\eta\in\RR$, be the characteristics defined by
\beq\label{RS}
\left\{\ba{ll}
\partial_x R(x,\xi)=\al\big(x,R(x,\xi)\big), & x\in\RR,
\\
R(0,\xi)=\xi
\ea\right.
\quad\mbox{and}\quad
\left\{\ba{ll}
\partial_x S(x,\eta)=\be\big(x,S(x,\eta)\big), & x\in\RR,
\\
S(0,\eta)=\eta
\ea\right.
\eeq
where $\al\neq\be$ are defined from the matrix-valued function $A$ of \eqref{A} by
\beq\label{albe}
\al:={b-\sqrt{a^2+b^2}\over a}\quad\mbox{and}\quad\be:={b+\sqrt{a^2+b^2}\over a}.
\eeq
Then, for $\ep$ small enough in \eqref{UMep}, there exist two functions $\xi,\eta\in C^4(\RR^2)$ satisfying
\beq\label{xieta}
\forall\,(x,y)\in\RR^2,\quad y=R\big(x,\xi(x,y)\big)=S\big(x,\eta(x,y)\big),
\eeq
and such that the mapping $(\xi,\eta):(x,y)\mapsto\big(\xi(x,y),\eta(x,y)\big)$ is a $C^4$-diffeomorphism on $\RR^2$ mapping $(0,0)$ to $(0,0)$.
\end{Lem}
\noindent
Using the change of variables
\beq\label{tz}
w(t,z)=u(x,y)\quad\mbox{where}\quad\left\{\ba{l} \dis t:={\xi(x,y)-\eta(x,y)\over 2} \\ \ecart \dis z:={\xi(x,y)+\eta(x,y)\over 2} \ea\right.
\quad\mbox{for }(x,y)\in\RR^2,
\eeq
and the following equalities, due to \eqref{RSxieta} below,
\beq\label{dxxidxeta}
\left\{\ba{l}
\partial_x\xi(x,y)=-\,\al\big(x,R(x,\xi(x,y))\big)\,\partial_y\xi(x,y)
\\ \ecart
\partial_x\eta(x,y)=-\,\beta\big(x,S(x,\eta(x,y))\big)\,\partial_y\eta(x,y)
\ea\right.
\quad\mbox{for }(x,y)\in\RR^2,
\eeq
a lengthy but classical computation leads us to
\beq\label{AD2u}
A:\nabla^2 u=\left({a^2+2b^2\over a}\,\partial_y\xi\,\partial_y\eta\right)\square w
+{1\over 2}\,A:(\nabla^2\xi-\nabla^2\eta)\,\partial_t w+{1\over 2}\,A:(\nabla^2\xi+\nabla^2\eta)\,\partial_z w.
\eeq
\par\ms\noindent
{\it Second step: Truncation of the coefficients.}
\par\noindent
To this end we need the following result.
\begin{Lem}\label{lem.trunc}
Let $k\in\NN$. Each periodic regular function $f\in C^{k+2}_\sharp(Y)$ can be linearly mapped to a function $\ph_f\in C^k_c(\RR^2)$ with compact support such that
\beq\label{phf}
\forall\,(x,y)\in\RR^2,\quad f(x,y)=\sum_{(p,q)\in\ZZ^2}\ph_f(x+p,y+q).
\eeq
Moreover, there is a constant $C_k>0$ only depending on $k$ such that
\beq\label{estphfk}
\forall\,f\in C^{k+2}_\sharp(Y),\quad \|\ph_f\|_{C^k(\RR^2)}\leq C_k\,\|f\|_{C^{k+2}_\sharp(Y)}.
\eeq
\end{Lem}
\noindent
Using the notation of Lemma~\ref{lem.trunc}, for any integer $n\in\NN$ and any function $f\in C^0_\sharp(Y)$, denote by $[f]_n$ the function defined by the truncation deduced from \eqref{phf},
\beq\label{[f]n}
[f]_n(x,y):=\sum_{p,q=-n}^{n}\ph_f(x+p,y+q)\quad\mbox{for }(x,y)\in\RR^2,
\eeq
so that $[f]_0=\ph_f$ and $[f]_\infty=f$.
\par
Now, return to the nonlinear wave equation \eqref{slwequ}.
Due to the compact supports of the functions $\ph_A$, $\ph_{\De U}$, $\ph_{\De({\rm curl}\,U)}$, for any $R>0$ there exists a smallest integer $n_R\in\NN$ such that
\beq\label{f=[f]n}
A=[A]_{n_R},\quad \De U=[\De U]_{n_R},\quad \De({\rm curl}\,U)=[\De({\rm curl}\,U)]_{n_R}\quad\mbox{in }D(0,R).
\eeq
Then, replacing the functions $\ph_A$, $\ph_{\De U}$, $\ph_{\De({\rm curl}\,U)}$ in the right-hand side of \eqref{slwequ} by their truncations at the size $n\in\NN$, we get the semilinear wave equation
\beq\label{slwequn}
A:\nabla^2 u_n=-[A]_n\nabla u_n\cdot\nabla u_n+R_\perp[\De U]_n\cdot\nabla u_n-{1\over 2}\,[\De(\curl\,U)]_n\quad\mbox{in }\RR^2.
\eeq
Next, using the change of variables \eqref{tz} and making a truncation of the coefficients of \eqref{AD2u} we obtain the modified equation satisfied by the function $w_n(t,z)=u_n(x,y)$,
\beq\label{slwequnwn}
\ba{l}
\dis \left({a^2+2b^2\over a}\,\partial_y\xi\,\partial_y\eta\right)\square w_n
+{1\over 2}\,[A]_n:(\nabla^2\xi-\nabla^2\eta)\,\partial_t w_n+{1\over 2}\,[A]_n:(\nabla^2\xi+\nabla^2\eta)\,\partial_z w_n
\\ \ecart
\dis =-[A]_n\nabla u_n\cdot\nabla u_n+R_\perp[\De U]_n\cdot\nabla u_n-{1\over 2}\,[\De(\curl\,U)]_n\quad\mbox{in }\RR^2
\ea
\eeq
Multiplying equation \eqref{slwequnwn} by the factor $a(a^2+2b^2)^{-1}\,\partial_{\xi}R\,\partial_{\eta}S$ which does not vanish in $\RR^2$ (due to \eqref{ane0} and to \eqref{RSxieta} below), we are led to the semilinear wave equation
\beq\label{slweqwn}
\left\{\ba{ll}
\square w_n=B_n\nabla w_n\cdot\nabla w_n+V_n\cdot\nabla w_n+h_n, & \mbox{in }\RR^2
\\ \ecart
w_n(0,z)=\partial_t w_n(0,z)=0, & z\in\RR.
\ea\right.
\eeq
The functions $B_n,V_n,h_n$ have coefficients in $C^2_c(\RR^2)$ since they can be expressed from \eqref{slwequnwn} in terms of the truncated functions $[A]_n,[\De U]_n,[\De(\curl\,U)]_n$. For example, we have
\beq\label{hn}
h_n(t,z):=-\,{a\,\partial_{\xi}R\big(x,\xi(x,y)\big)\,\partial_{\eta}S\big(x,\eta(x,y)\big)\over 2\,(a^2+2b^2)}\,[\De(\curl\,U)]_n(x,y)
\quad\mbox{for }(x,y)\in\RR^2,
\eeq
which is in $C^2(\RR^2)$ since $a,b,R,S,\xi,\eta\in C^4(\RR^2)$ and $U\in C^5(\RR^2)^2$.
Moreover, since $[\De(\curl\,U)]_n$ has compact support and the mapping $(x,y)\mapsto (t,z)$ defined by \eqref{xieta}, \eqref{tz} is proper (see estimate \eqref{estxyxieta} in the proof of Lemma~\ref{lem.canonic} below), the function $h_n$ has also compact support with respect to the new variables $(t,z)$.
\par\ms\noindent
{\it Third step: Conclusion thanks to Lemma~\ref{lem.G}.}
\par\noindent
Let $R>0$ and let $n=n_R$ be the integer such that the equalities \eqref{f=[f]n} hold.
Since $B_n,V_n,h_n$ have coefficients in $C^2_c(\RR^2)$, the condition \eqref{H1} is fulfilled with $B_n$ and $V_n$.
By the estimate~\eqref{estphfk} and the definition \eqref{[f]n} there exists a constant $C>0$ such that
\[
\big\|[\De(\curl\,U)]_n\big\|_{C^0(\RR^2)}\leq C\,\big\|\De(\curl\,U)\big\|_{C^2_\sharp(Y)}\leq 4\,C\,\big\|U(X)-MX\big\|_{C^5_\sharp(Y)^{2\times 2}}.
\]
This combined with estimate \eqref{UMep} and the definition \eqref{hn} of $h_n\in C^2_c(\RR^2)$, implies the existence of a constant $C_n>0$ such that
\[
\int_{\RR}\|h_n(t,.)\|_{L^\infty(\RR)}\,dt<C_n\,\ep.
\]
Since by construction the functions $B_n,V_n$ have compact support, we also have
\[
\La_n:=\int_\RR\left(\|B_n(t,\cdot)\|_{L^\infty(\RR)}+\|V_n(t,\cdot)\|_{L^\infty(\RR)}\right)dt<\infty.
\]
Then, by virtue of Lemma~\ref{lem.G}, choosing $\ep=\ep_R>0$ small enough, there exists a global solution $w_n=w_{n_R}\in C^2(\RR^2)$ to the equation \eqref{slweqwn}, or equivalently, a solution $u_n=u_{n_R}\in C^2(\RR^2)$ to the equation \eqref{slwequn}. Finally, using the equalities \eqref{f=[f]n} the function $u_{n_R}$ is a solution of the initial wave equation \eqref{slwequ} in the disk $D(0,R)$, and the function $\mu_R:=e^{u_{n_R}}$ solves the equation \eqref{Stokes2} in $\Om=D(0,R)$. Therefore, the field $e(U)$ is isotropically realizable in $D(0,R)$ with the positive viscosity $\mu_R$.
\par
It remains to prove the global realizability in $\RR^2$ under the boundedness condition on $\ep_R$. We have $\mu_R=e^{u_R}\in C^2\big(D(0,R)\big)$, and $u_R(x,y)=w_n(t,z)$ for $(x,y)\in D(0,R)$ (thanks to the change of variables \eqref{tz}), where $w_n=w_{n_R}$ is the unique (see \cite{Hor}, Theorem~6.4.10) solution to the semilinear wave equation \eqref{slweqwn}. Hence, if $S>R>0$ and the inequality \eqref{UMep} is satisfied with $\ep\leq \min(\ep_R,\ep_S)$, then $n_R\leq n_S$ and by the uniqueness in equation \eqref{slweqwn} combined with the following equalities (due to \eqref{f=[f]n})
\[
B_{n_S}(t,z)=B_{n_R}(t,z),\quad V_{n_S}(t,z)=V_{n_R}(t,z),\quad h_{n_S}(t,z)=h_{n_R}(t,z)\qquad\mbox{for }(x,y)\in D(0,R),
\]
we get that $w_{n_R}(t,z)=w_{n_S}(t,z)$ for $(x,y)\in D(0,R)$, so that $\mu_S=\mu_R$ in $D(0,R)$.
\par
Finally, assume that $\ep_R$ is bounded from below by a positive constant $\ep_\infty$ independent of $R$.
Under the perturbation condition \eqref{UMep} with $\ep=\ep_\infty$, we can define the function $\mu\in C^2(\RR^2)$ by $\mu:=\mu_R$ in $D(0,R)$, for any $R>0$.
The function $\mu$ clearly satisfies the equation
\[
\curl\left[\Div\,\big(\mu\,e(U)\big)\right]=0\quad\mbox{in }\RR^2.
\]
Therefore, the field $e(U)$ is isotropically realizable in the whole space $\RR^2$.
\cqfd
\par\bigskip
The isotropic realizability in the torus is more intricate since it is connected to the existence of a bounded global solution to the equation \eqref{slwequ}. To illuminate this we have the following result.
\begin{Pro}\label{pro.alt}
Let $U\in C^3(\RR^2)^2$ be a divergence free function such that $DU$ is $Y$-periodic.
Assume that $e(U)$ is not isotropically realizable in the torus.
Then,  we have the following alternative:
\begin{enumerate}
\item the semilinear equation \eqref{slwequ} has not a global regular solution,
\item any global regular solution to \eqref{slwequ} is either not bounded or not uniformly continuous in~$\RR^2$.
\end{enumerate}
\end{Pro}
\begin{proof}
Assume by contradiction that $u_0$ is a regular solution to equation \eqref{slwequ}, which is bounded and uniformly continuous in $\RR^2$.
Then, the function $\mu_0:=e^{u_0}$ is uniformly continuous in $\RR^2$, and the sequence $(\mu_k)_{k\geq 1}$ defined by
\beq\label{muk}
\mu_k(x,y):={1\over (2k+1)^2}\,\sum_{p,q=-k}^k \mu_0(x+p,y+q)\quad\mbox{for }(x,y)\in\RR^2,
\eeq
is uniformly bounded and equi-continuous in $\RR^2$. Hence, by virtue of Ascoli's theorem $\mu_k$ converges uniformly in $\RR^2$ to some function $\mu\in C^0(\RR^2)$ up to a subsequence. Note that $\mu$ is bounded from below in $\RR^2$ by a positive constant. Moreover, since we have for any $n\geq 1$,
\[
\big|\mu_k(x+1,y)-\mu_k(x,y)\big|+\big|\mu_k(x,y+1)-\mu_k(x,y)\big|\leq{4\,\|\mu_0\|_{L^\infty(\RR^2)}\over 2k+1}\quad\mbox{for }(x,y)\in\RR^2,
\]
it follows that $\mu$ is $Y$-periodic. Next, using that $u_0$ is solution to equation \eqref{slwequ} and the $Y$-periodicity of $DU$, we get that
\[
\curl\left[\Div\,\big(\mu_k\,e(U)\big)\right]=0\quad\mbox{in }\RR^2.
\]
Passing to the limit as $k\to\infty$, it follows that at least in the distributions sense
\beq\label{isoreaDUdis}
\curl\left[\Div\,\big(\mu\,e(U)\big)\right]=0\quad\mbox{in }\D'(\RR^2),
\eeq
which implies that $e(U)$ is isotropically realizable in the torus with the positive periodic continuous function $\mu$.
\end{proof}
\begin{Rem}
If we relax the continuity condition on the viscosity in the definition of the isotropic realizability, the alternative of Proposition~\ref{pro.alt} reduces to
\begin{enumerate}
\item the semilinear equation \eqref{slwequ} has not a global regular solution,
\item any global regular solution to \eqref{slwequ} is not bounded in~$\RR^2$.
\end{enumerate}
Indeed, in the proof of Proposition~\ref{pro.alt} the sequence of regular functions $\mu_k$ now converges weakly-$\ast$ in $L^\infty(\RR^2)$ to some function $\mu$ which is not necessarily continuous but still periodic. Therefore, the limit equation \eqref{isoreaDUdis} remains satisfied, which again allows us to conclude.
\end{Rem}
\subsection{Proofs of the technical lemmas}
\noindent
{\bf Proof of Lemma~\ref{lem.canonic}.}
Let $(x,y)\in\RR^2$. Since for any $\xi\in\RR$,
\[
\partial_{\xi}R(x,\xi)=\exp\left(\int_0^x\partial_y\al\big(t,R(t,\xi)\big)\,dt\right)>0\quad\mbox{and}\quad\big|R(x,\xi)-\xi\big|\leq\|\al\|_{L^\infty(\RR^2)}\,|x|,
\]
the mapping $\xi\mapsto R(x,\xi)$ is a $C^1$-diffeomorphism on $\RR$. Hence, there exists a unique $\xi(x,y)\in\RR$ such that $y=R\big(x,\xi(x,y)\big)$.
Again using that $\partial_{\xi}R\neq 0$ and that $R\in C^4(\RR^2)$ (recall that $U\in C^5(\RR^2)$), the implicit function theorem implies that $\xi\in C^4(\RR^2)$ . Similarly, the function $\eta$ defined implicitly by $y=S\big(x,\eta(x,y)\big)$ belongs to $C^4(\RR^2)$.
Moreover, by the chain rule applied to \eqref{xieta} using \eqref{RS} we get that
\beq\label{RSxieta}
\left\{\ba{l}
0=\al\big(x,R(x,\xi(x,y))\big)+\partial_\xi R\big(x,\xi(x,y)\big)\,\partial_x\xi(x,y)
\\ \ecart
0=\be\big(x,R(x,\xi(x,y))\big)+\partial_\eta S\big(x,\eta(x,y)\big)\,\partial_x\eta(x,y)
\\ \ecart
1=\partial_\xi R\big(x,\xi(x,y)\big)\,\partial_y\xi(x,y)
\\ \ecart
1=\partial_\eta S\big(x,\eta(x,y)\big)\,\partial_y\eta(x,y),
\ea\right.
\quad\mbox{for }(x,y)\in\RR^2.
\eeq
Hence, the Jacobian $J_{(\xi,\eta)}$ of the mapping $(\xi,\eta)$ satisfies
\[
J_{(\xi,\eta)}=(\be-\al)\,\partial_y\xi\,\partial_y\eta={\be-\al\over\partial_{\xi}R\,\partial_{\eta}S}\neq 0\quad\mbox{in }\RR^2.
\]
Let us now prove that the mapping $(\xi,\eta)$ is proper, {\em i.e.} the reciprocal of any compact set $K$ in $\RR^2$ is compact in $\RR^2$.
Consider $(x,y)\in\RR^2$ such that $\big(\xi(x,y),\eta(x,y)\big)\in K$. By \eqref{RS} and \eqref{xieta} we have
\[
\xi(x,y)-\eta(x,y)=\int_0^x\big[\partial_t S(t,\eta)-\partial_t R(t,\xi)\big]\,dt=\int_0^x\left[\be\big(t,S(t,\eta)\big)-\al\big(t,R(t,\xi)\big)\right]dt.
\]
This combined with the definition \eqref{albe} of $\al,\be$ and the estimates \eqref{UMep}, \eqref{ane0} satisfied by $a$, $b$, implies that for $\ep$ small enough,
\[
\big|\xi(x,y)-\eta(x,y)\big|\geq c_K\,|x|,
\]
where $c_K>0$ only depends on $K$.
Moreover, by \eqref{xieta} and \eqref{RS} we have
\[
\big|y-\xi(x,y)\big|=\left|R\big(x,\xi(x,y)\big)-R\big(0,\xi(x,y)\big)\right|\leq C_K\,|x|,
\]
where $C_K>0$ only depends on $K$.
The two former estimates yield that
\beq\label{estxyxieta}
|x|+|y|\leq\big|\xi(x,y)\big|+{C_K+1\over c_K}\,\big|\xi(x,y)-\eta(x,y)\big|,
\eeq
which shows that $(x,y)$ lies in a compact set of $\RR^2$.
By virtue of Hadamard's theorem the properness of the $C^4$-mapping $(\xi,\eta)$ and the non-vanishing of its Jacobian imply that it is a $C^4$-diffeomorphism on $\RR^2$.
\cqfd
\par\bs\noindent
{\bf Proof of Lemma~\ref{lem.trunc}.}
Let $f\in C^{k+2}_\sharp(Y)$, $k\in\NN$.
Fix $\th\in C^\infty_c(\RR^2)$ with $\int_{\RR^2}\th(X)\,dX=1$.
Define $h$ by the convolution $h:=\th\ast 1_{Y}$ (recall that $Y=[0,1]^2$), and the function $\ph_f$ by
\beq\label{phfh}
\ph_f(x,y):=\sum_{(p,q)\in\ZZ^2}\hat{f}(p,q)\,e^{2i\pi(px+qy)}\,h(x,y)\quad\mbox{for }(x,y)\in\RR^2,
\eeq
where $\hat{f}(p,q)$ denotes the Fourier coefficient of $f$ given by
\[
\hat{f}(p,q):=\int_Y f(x,y)\,e^{-2i\pi\,(px+qy)}\,dx\,dy\quad\mbox{for }(p,q)\in\ZZ^2.
\]
Using that
\[
\widehat{\partial_x f}(p,0)=2i\pi p\hat{f}(p,0),\quad \widehat{\partial_y f}(0,q)=2i\pi q\hat{f}(0,q),\quad
\widehat{\partial^2_{xy} f}(p,q)=-4\pi^2 pq\hat{f}(p,q),
\]
Cauchy-Schwarz' inequality and Parseval's identity imply that
\[
\ba{l}
\dis \sum_{(p,q)\in\ZZ^2}\big|\hat{f}(p,q)\big|\leq\big|\hat{f}(0,0)\big|+\sum_{p\in\ZZ\setminus\{0\}}{\big|\widehat{\partial_x f}(p,0)\big|\over 2\pi|p|}
+\sum_{q\in\ZZ\setminus\{0\}}{\big|\widehat{\partial_y f}(0,q)\big|\over 2\pi|q|}
+\sum_{(p,q)\in(\ZZ\setminus\{0\})^2}{\big|\widehat{\partial_{xy}f}(p,q)\big|\over 4\pi^2|pq|}
\\ \ecart
\leq c\,\|f\|_{L^2(Y)}+c\,\|\partial_x f(\cdot,0)\|_{L^2(Y)}+c\,\|\partial_y f(0,\cdot)\|_{L^2(Y)}
+c\,\|\partial^2_{xy}f\|_{L^2(Y)}\leq C\,\|f\|_{C^{2}_\sharp(Y)}.
\ea
\]
Hence, it follows that $\ph_f\in C^0_c(\RR^2)$ and estimate \eqref{estphfk} is satisfied for $k=0$.
Iterating we get that $\ph_f\in C^k_c(\RR^2)$ and \eqref{estphfk} holds for any $k\in\NN$.
\par
Now, consider the function $g$ defined by
\[
g(x,y):=\sum_{(p,q)\in\ZZ^2}\ph_f(x+p,y+q)\quad\mbox{for }(x,y)\in\RR^2,
\]
which is clearly in $C^k_\sharp(Y)$.
We also have $\hat{g}(p,q)=\F(\ph_f)(p,q)$ for any $(p,q)\in\ZZ^2$, where $\F$ denotes the Fourier transform in $L^1(\RR^2)$.
On the other hand, by the definition \eqref{phfh} of $\ph_f$ we have
\[
\ba{ll}
\F(\ph_f)(p,q) & \dis =\sum_{(j,k)\in\ZZ^2}\hat{f}(j,k)\left(\int_{\RR^2}e^{2i\pi\,[(j-p)x+(k-q)y]}\,h(x,y)\,dx\,dy\right)
\\ \ecart
& \dis =\sum_{(j,k)\in\ZZ^2}\hat{f}(j,k)\,\F(h)(p-j,q-k),
\ea
\]
and since $\F(\th)(0,0)=1$,
\[
\F(h)(p-j,q-k)=\F(\th)(p-j,q-k)\,\F(1_Y)(p-j,q-k)=\left\{\ba{ll} 1 & \mbox{if }(j,k)=(p,q) \\ 0 & \mbox{if }(j,k)\neq (p,q).\ea\right.
\]
Therefore, we get that $\hat{g}(p,q)=\hat{f}(p,q)$ for any $(p,q)\in\ZZ^2$, which implies that $g=f$ and thus the representation formula \eqref{phf}.
\cqfd
%%%%%%%%%%
\section{A few singular cases}\label{s.exa}
\subsection{An example of non-isotropic realizability in the torus}\label{ss.cex}
Let $U_\ep$, $\ep>0$, be the divergence free vector-valued function defined in $\RR^2$ by
\[
U_\ep(x,y):=\begin{pmatrix} \dis x-{\ep\over\pi}\,\cos(2\pi y) \\ -\,y\end{pmatrix}\quad\mbox{for }(x,y)\in\RR^2.
\]
The associated field
\beq\label{eUnirt}
e(U_\ep)
%={1\over 2}\left(DU_\ep+^{t}\!DU_\ep\right)
=\begin{pmatrix} 1 & \ep\sin(2\pi y) \\ \ep\sin(2\pi y) & -1 \end{pmatrix}\quad\mbox{for }(x,y)\in\RR^2,
\eeq
is a smooth and $Y$-periodic perturbation of the constant matrix ${\rm diag}\,(1,-1)$.
We have the following result of non-realizability:
\begin{Pro}\label{pro.norea}
The periodic field $e(U_\ep)$ defined by \eqref{eUnirt} is isotropically realizable in $\RR^2$, but not in the torus.
Moreover, the semilinear wave equation \eqref{slwequ} associated with $U=U_\ep$ has a global regular solution, and any global regular solution to \eqref{slwequ} is not bounded or not uniformly continuous in $\RR^2$.
\end{Pro}
\begin{proof}
Assume by contradiction that $e(U_\ep)$ is isotropically realizable with a continuous positive conductivity $\mu(x,y)$ and a continuous pressure $p(x,y)$ which are both $1$-periodic with respect to the variable $x$. Set $Q:=(0,1)\times(-r,r)$, with $r\in(0,1/2)$. Integrating by parts and using the periodicity with respect to $x$, we get that ($\nu$ denotes the outside normal to $\partial Q$)
\[
\ba{ll}
0 & \dis =\int_Q\big[\Div\left(\mu\,e(U_\ep)\right)-\nabla p\big]\cdot e_x\,dx\,dy=\int_{\partial Q} \mu\,e(U_\ep):(e_x\otimes \nu)\,ds+0
\\ \ecart
& \dis =\int_0^1\left[\big(\mu\,e(U_\ep)\big)(x,r):(e_x\otimes e_y)-\big(\mu\,e(U_\ep)\big)(x,-r):(e_x\otimes e_y)\right]dx
\\ \ecart
& \dis =\ep\sin(2\pi r)\int_0^1\big[\mu(x,r)+\mu(x,-r)\big]\,dx>0.
\ea
\]
This contradiction shows that for any $\ep>0$, the periodic field $e(U_\ep)$ is not isotropically realizable in the torus as a strain field.
On the contrary, the isotropic realizability holds clearly for $\ep=0$, since $e(U_0)={\rm diag}\,(1,-1)$.
\par
The semilinear wave equation \eqref{slwequ} associated with $U_\ep$ reads as
\beq\label{slwequep}
\ba{l}
2\,\partial^2_{xy}u-\ep\sin(2\pi y)\,\partial^2_{xx}u+\ep\sin(2\pi y)\,\partial^2_{yy}u
\\ \ecart
=-\,2\,\partial_{x}u\,\partial_{y}u+\ep\sin(2\pi y)\,(\partial_{x}u)^2-\ep\sin(2\pi y)\,(\partial_{y}u)^2
\\ \ecart
+\,4\pi\ep\cos(2\pi y)\,\partial_{y}u-4\pi^2\ep\sin(2\pi y),
\ea
\eeq
which clearly has $u(x,y)=2\pi x$ as a global solution in $\RR^2$.
However, by virtue of Proposition~\ref{pro.alt} we have the following alternative:
\begin{enumerate}
\item equation \eqref{slwequep} has not a global regular solution,
\item any global regular solution to \eqref{slwequep} is either not bounded or not uniformly continuous in~$\RR^2$.
\end{enumerate}
Therefore, the second alternative holds for equation \eqref{slwequep}.
\end{proof}
\subsection{An example with a vanishing field}\label{ss.exavan}
The case where the strain field $e(U)$ vanishes at one point $X_*$ is more delicate. For the moment we have no general result. However, the following case with separate variables already shows the difficulties of the problem.
\begin{Pro}\label{pro.sepvar}
Let $f$ and $g$ be two functions in $C^0([-1,1])$ satisfying
\beq\label{fg}
f(0)=g(0)=0\quad\mbox{and}\quad\forall\,x\in[-1,1]\setminus\{0\},\;\;f(x)>0,\ g(x)>0,
\eeq
and having asymptotic expansions of any order at the point $0$.
\par\noindent
Consider the strain field $e(U)$ defined by
\beq\label{e(U)fg}
e(U):=\begin{pmatrix} 0 & f(x)+g(y) \\ f(x)+g(y) & 0 \end{pmatrix}\quad\mbox{for }(x,y)\in[-1,1]^2.
\eeq
Then, a necessary and sufficient condition for $e(U)$ to be isotropically realizable for the incompressible Stokes equation with a continuous function $\mu>0$  in a neighborhood of $(0,0)$, is that there exists $a>0$ such that
\beq\label{expfg0}
f(x)=a\,x^2+o(x^2)\quad\mbox{and}\quad g(x)=a\,x^2+o(x^2).
\eeq
\end{Pro}
\begin{Rem}
The strain field $e(U)$ defined by \eqref{e(U)fg} is for example associated with the divergence free field given by
\beq\label{Ufg}
U(x,y)=2 \begin{pmatrix}\dis \int_0^y g(t)\,dt \\\ecart \dis \int_0^x f(t)\,dt \end{pmatrix}\quad\mbox{for }(x,y)\in[-1,1]^2.
\eeq
Moreover, due to condition \eqref{fg} the strain field $e(U)$ only vanishes at the point $(0,0)$ in the neighborhood of this point.
\end{Rem}
\begin{Rem}
If we relax the continuity assumption of $\mu$ at the point $(0,0)$, then the necessary and sufficient condition of realizability becomes
\beq\label{expfg0dis}
\exists\, n\in 2\NN,\ \exists\,a,b>0,\quad f(x)=a\,x^n+o(x^n)\quad\mbox{and}\quad g(x)=b\,x^n+o(x^n).
\eeq
This is induced by the third step of the proof of Proposition~\ref{pro.sepvar} below.
\end{Rem}
\noindent
{\bf Proof of Proposition~\ref{pro.sepvar}.} The proof is divided in three steps according to the expansions of $f,g$ at the point $0$. In the sequel, for any nonnegative functions $\ph,\psi$ being continuous in a neighborhood of $(0,0)$, we denote $\ph\approx\psi$ when there exists a constant $c>1$ such that $c^{-1}\,\ph\leq\psi\leq c\,\ph$ in a neighborhood of~$(0,0)$.
\par\ms\noindent
{\it First case: }$\forall\,k\in\NN$, $f(x)=o(x^k)$ and $g(x)=o(x^k)$.
\par\noindent
Assume that $e(U)$ is realizable with a positive continuous viscosity $\mu$ on a non-empty open disk $\Om$ centered on $(0,0)$. Then, there exists a pressure $p\in L^2(\Om)$ such that the Stokes equation~\eqref{Stokes} holds. Hence, we have
\beq
\partial_y\big(\mu\left(f(x)+g(y)\right)\big)=\partial_x p\quad\mbox{and}\quad\partial_x\big(\mu\left(f(x)+g(y)\right)\big)=\partial_y p 
\quad\mbox{in }\Om,
\eeq
which implies that
\beq
\partial^2_{xx}\big(\mu\left(f(x)+g(y)\right)\big)-\partial^2_{yy}\big(\mu\left(f(x)+g(y)\right)\big)=0\quad\mbox{in }\Om.
\eeq
Therefore, there exist two continuous functions $F,G$ defined around the point $0$ such that
\beq\label{muFG}
\mu(x,y)\left(f(x)+g(y)\right)=F(x+y)+G(x-y)\quad\mbox{in a neighborhood of $(0,0)$}.
\eeq
Since $\mu$ is positive and continuous in the neighborhood of $(0,0)$, the previous equality yields
\beq\label{FG}
f(x)+g(y)\approx F(x+y)+G(x-y).
\eeq
We have $F(0)+G(0)=0$ so that we can assume that $F(0)=G(0)=0$ replacing $F$ and $G$ by $F-F(0)$ and $G-G(0)$.
Taking successively $y=0$, $x=0$, $y=x$ and $y=-x$ in \eqref{FG}, we get that
\beq\label{fgFG}
\left\{\ba{rr}
f(x)\approx F(x)+G(x), & g(x)\approx F(x)+G(-x),
\\ \ecart
F(2x)\approx f(x)+g(x),  & G(2x)\approx f(x)+g(-x),
\ea\right.
\eeq
which implies that
\beq
f(x)\approx f(x/2)+g(x/2)+g(-x/2)\quad\mbox{and}\quad g(x)\approx f(x/2)+f(-x/2)+g(x/2).
\eeq
Hence, the even nonnegative function $h$ defined by $h(x):=f(x)+f(-x)+g(x)+g(-x)$ satisfies
\beq\label{h}
h(x)\approx h(x/2)\quad\mbox{and}\quad\forall\,k\in\NN,\;\;h(x)=x^k\,\ep_k(x)\;\;\mbox{with }\lim_{x\to 0}\ep_k(x)=0.
\eeq
Reiterating the first condition of \eqref{h} there exists $c>1$ such that
\beq
\forall\,n,k\in\NN,\quad h(x)\leq c^n\,h(2^{-n}x)=c^n\,2^{-kn}\,x^k\,\ep_k(2^{-n}x),\quad\mbox{for $x>0$ close to $0$}.
\eeq
Then, choosing $k\in\NN$ such that $2^{-k}c\leq 1$, it follows that
\beq
0\leq h(x)\leq\lim_{n\to\infty}(2^{-k}c)^n\,x^k\,\ep_k(2^{-n}x)=0,
\eeq
which yields $h(x)=0$ and contradicts the assumption \eqref{fg} on $f,g$. Therefore, the isotropic realizability cannot be satisfied in this case.
\par\ms\noindent
{\it Second case: }$\exists\,m\in\NN,\ \exists\,a\neq 0$, $f(x)=a\,x^m+o(x^m)\ $ and $\ \forall\,k\in\NN$, $g(x)=o(x^k)$.
\par\noindent
By \eqref{fg} $m$ is an even positive integer and $a>0$. This combined with \eqref{fgFG} yields
\beq
\left\{\ba{ll}
F(x)+G(-x)\approx g(x),
\\ \ecart
F(x)\approx a\,(x/2)^m+g(x/2)\approx x^m, & G(-x)\approx a\,(-x/2)^m+g(x/2)\approx x^m.
\ea\right.
\eeq
Hence, we obtain that
\beq
x^m\approx F(x)+G(-x)\approx g(x),
\eeq
which leads us to a contradiction. Therefore, the isotropic realizability cannot hold in this case.
\par\ms\noindent
{\it Third case: }$\exists\,m,n\in\NN,\ \exists\,a,b\neq 0$, $f(x)=a\,x^m+o(x^m)$ and $g(x)=b\,x^n+o(x^n)$.
\par\noindent
By \eqref{fg} $m,n$ are even positive integers and $a,b>0$. Again by \eqref{fgFG} we have
\beq
F(x)+G(x)\approx x^m,\quad F(x)+G(-x)\approx x^n,\quad F(x)\approx x^m+x^n,\quad G(\pm\,x)\approx x^m+x^n.
\eeq
Hence, we get that
\beq
x^m\approx x^m+x^n\approx x^n,
\eeq
which implies that $m=n$.
\par
On the other hand, the continuity of $\mu$ \eqref{muFG} at $(0,0)$ implies that
\beq\label{limmu}
\lim_{(x,y)\to(0,0)}{F(x+y)+G(x-y)\over f(x)+g(y)}=\ell:=\mu(0,0)>0.
\eeq
Taking successively $y=x$, $y=-x$, $y=0$ and $x=0$ in \eqref{limmu}, we get that
\beq
\left\{\ba{rl}
\dis F(2x)\,\mathop{\sim}_{0}\,\ell\left(a+b\right)x^n, & \dis G(2x)\,\mathop{\sim}_{0}\,\ell\left(a+b\right)x^n,
\\ \ecart
\dis F(x)+G(x)\,\mathop{\sim}_{0}\,\ell\,a\,x^n, & \dis F(x)+G(-x)\,\mathop{\sim}_{0}\,\ell\,b\,x^n,
\ea\right.
\eeq
which implies that $2(a+b)=2^n a=2^n b$.
Therefore, we deduce that $a=b$ and $n=2$.
\par\ms
Conversely, if the asymptotic expansions \eqref{expfg0} hold, then the function $\mu$ defined by
\beq
\mu(x,y):=\left\{\ba{cl}
\dis {x^2+y^2\over f(x)+g(y)} & \mbox{if }(x,y)\in[-1,1]^2\setminus\{(0,0)\}
\\ \ecart
\dis {1\over a} & \mbox{if }(x,y)=(0,0),
\ea\right.
\eeq
is positive and continuous in $[-1,1]^2$. Moreover, by \eqref{e(U)fg} we have
\beq
\Div\,\big(\mu\,e(U)\big)=\Div\begin{pmatrix} 0 & x^2+y^2 \\ x^2+y^2 & 0 \end{pmatrix}=\nabla\left(2\,xy\right).
\eeq
Therefore, the strain field $e(U)$ is realizable with the continuous positive function $\mu$ in a neighborhood of $(0,0)$, which concludes the proof of Proposition~\ref{pro.sepvar}.
\cqfd
\subsection{The case of a laminate field}\label{ss.lam}
Theorem~\ref{thm.isorea} does not hold under the sole condition \eqref{e(U)­0} for non-regular fields.
An example of this situation is given by the laminates.
\begin{Def}
A two-phase rank-one laminate field is any measurable function $E:\RR^2\to\RR^{2\times 2}$ defined by
\beq\label{E}
E(X):=\chi(\xi\cdot X)\,E_1+\big(1-\chi(\xi\cdot X)\big)\,E_2\quad\mbox{for a.e. }X=(x,y)\in\RR^2,
\eeq
where $E_1,E_2$ are two given matrices in $\RR^{2\times 2}$, $\xi$ is a given unit norm vector in $\RR^2$, and $\chi$ is a characteristic function in $L^\infty(\RR)$.
\end{Def}
We have the following characterization of rank-one laminates:
\begin{Pro}
A two-phase rank-one laminate field $E$ of type \eqref{E} is a strain field $e(U)$ for some divergence free Lipschitz function $U:\RR^2\to\RR^2$, if and only if
\beq\label{strfieldlam}
E_1,E_2\in\RR^{2\times 2}_{s,0}\quad\mbox{and}\quad\exists\,\la\in\RR,\;\;E_1-E_2=\la\,\xi\odot R_{\perp}\xi.
\eeq
\end{Pro}
\begin{proof}
Assume that the field $E$ of \eqref{E} agrees with $e(U)$ for some divergence free Lipschitz function $U:\RR^2\to\RR^2$. Since the strain field $e(U)$ is a zero trace symmetric matrix-valued function, the phases $E_1,E_2$ belong to $\RR^{2\times 2}_{s,0}$ (symmetric with zero trace).
Moreover, the strain field $e(U)$ satisfies the differential constraint
\beq
\partial^2_{xx}\big[e(U)_{22}\big]+\partial^2_{yy}\big[e(U)_{11}\big]=2\,\partial^2_{xy}\big[e(U)_{12}\big]\quad\mbox{in }\RR^2,
\eeq
which by \eqref{E} implies that
\beq
\left(\xi_x^2-\xi_y^2\right)(E_1-E_2)_{22}=-\,2\,\xi_x\xi_y\,(E_1-E_2)_{12}.
\eeq
Since $\xi_x^2-\xi_y^2$ and $\xi_x\xi_y$ are not simultaneously zero ($|\xi|=1$) and $E_1-E_2\in\,\RR^{2\times 2}_{s,0}$, we deduce from the previous equality the existence of $\la\in\RR$ such that
\beq
E_1-E_2=\la\begin{pmatrix} -2\,\xi_x\xi_y & \xi_x^2-\xi_y^2 \\ \xi_x^2-\xi_y^2 & 2\,\xi_x\xi_y \end{pmatrix}
=\la\,\xi\odot R_{\perp}\xi.
\eeq
\par
Conversely, assume that \eqref{strfieldlam} holds. Consider the Lipschitz function $U$ defined by
\beq
U(X):=E_2\,X+\la\left(\int_0^{\xi\cdot X}\chi(t)\,dt\right)R_{\perp}\xi\quad\mbox{for }X\in\RR^2.
\eeq
We have
\beq
DU(X)=E_2+\la\,\chi(\xi\cdot X)\,\xi\otimes R_\perp\xi,\quad\mbox{for a.e. }X\in\RR^2,
\eeq
hence by \eqref{strfieldlam} and \eqref{E}
\beq
e(U)(X)=E_2+\chi(\xi\cdot X)\left(E_1-E_2\right)=E(X)\quad\mbox{for a.e. }X\in\RR^2,
\eeq
which concludes the proof.
\end{proof}
Now, define the isotropic realizability for laminates.
\begin{Def}\label{def.realam}
A two-phase rank-one laminate field $E$ of type \eqref{E} is isotropically realizable for the Stokes equation in $\RR^2$ if there exist a two-phase rank-one positive function $\mu$ defined by
\beq\label{mulam}
\mu(X):=\chi(\xi\cdot X)\,\mu_1+\big(1-\chi(\xi\cdot X)\big)\,\mu_2\quad\mbox{for a.e. }X\in\RR^2,\quad\mbox{with }\mu_1,\mu_2>0,
\eeq
and a function $p\in L^2_{\rm loc}(\RR^2)$ such that in the distributions sense,
\beq\label{StokesE}
-\,\Div\left(\mu\,E\right)+\nabla p=0\quad\mbox{in }\RR^2.
\eeq
\end{Def}
Then, we have the following realizability result for two-phase rank-one laminates.
\begin{Thm}\label{thm.realam}
A strain field $E$ of type \eqref{E} is isotropically realizable in the sense of Definition~\ref{def.realam}, if and only if
\beq\label{E1E2}
E_1:E_2>{|E_1|^2|E_2|^2+(E_1:E_2)^2\over |E_1|^2+|E_2|^2}\quad\mbox{or}\quad E_1=E_2.
\eeq
\end{Thm}
\begin{Rem}
When $E_1\neq E_2$, the realizability condition \eqref{E1E2} is stronger than the condition \eqref{e(U)­0}, {\em i.e.} $E\neq 0$, of the regular case. It depends only on the values of the strain field in the phases and not on the lamination direction.
\end{Rem}
\noindent
{\bf Proof of Theorem~\ref{thm.realam}.}
Consider a strain field $E$ of type \eqref{E} and a positive function $\mu$ of type~\eqref{mulam}. We have in the distributions sense 
\beq
\Div\left(\mu\,E\right)=\chi'(\xi\cdot X)\,(\mu_1\,E_1-\mu_2\,E_2)\,\xi\quad\mbox{in }\RR^2.
\eeq
Hence, if equation \eqref{StokesE} holds with a function $p\in L^2_{\rm loc}(\RR^2)$, then $\nabla p=\chi'(\xi\cdot X)\,q$ in $\RR^2$ for some fixed vector $q\in\RR^2$, which implies that $q\parallel \xi$. Thus, $p$ is also a two-phase rank-one laminate function, {\em i.e.}
\beq\label{plam}
p(x)=\chi(\xi\cdot X)\,p_1+\big(1-\chi(\xi\cdot X)\big)\,p_2\quad\mbox{for a.e. }X\in\RR^2.
\eeq
It follows that $E$ solves the Stokes equation \eqref{StokesE} with $\mu$ and $p$ if and only if
\beq
(\mu_1\,E_1-\mu_2\,E_2)\,\xi=(p_1-p_2)\,\xi.
\eeq
Therefore, $E$ is isotropically realizable in the sense of Definition~\ref{thm.realam} if and only if there exist two constants $\mu_1,\mu_2>0$ such that $(\mu_1\,E_1-\mu_2\,E_2)\,\xi\parallel\xi$, or equivalently
\beq\label{E1E2a}
\exists\,\mu_1,\mu_2>0,\quad\mu_1\,E_1\,R_\perp\xi\cdot\xi=\mu_2\,E_2R_\perp\xi\cdot\xi.
\eeq
\par
Next, let us check that condition \eqref{E1E2a} is equivalent to the following one
\beq\label{E1E2b}
\left(E_1R_{\perp}\xi\cdot\xi\right)\left(E_2R_{\perp}\xi\cdot\xi\right)>0\quad\mbox{or}\quad E_1=E_2.
\eeq
It is clear that condition \eqref{E1E2b} implies \eqref{E1E2a}. 
Conversely, if \eqref{E1E2a} holds then
\beq\label{E1E2c}
\left(E_1R_{\perp}\xi\cdot\xi\right)\left(E_2R_{\perp}\xi\cdot\xi\right)\geq 0.
\eeq
To obtain \eqref{E1E2b} it is enough to deal with the case of equality in \eqref{E1E2c}.
This combined with \eqref{E1E2a} implies that $E_1R_{\perp}\xi\cdot\xi=E_2R_{\perp}\xi\cdot\xi=0$.
However, by the jump condition of \eqref{strfieldlam} we have
\beq
E_1R_{\perp}\xi\cdot\xi-E_2R_{\perp}\xi\cdot\xi
={\la\over 2}\,\big[(\xi\otimes R_\perp\xi)R_\perp\xi\cdot\xi+(R_\perp\xi\otimes\xi)R_\perp\xi\cdot\xi\big]
={\la\over 2}\,|R_\perp\xi|^2\,|\xi|^2={\la\over 2},
\eeq
which yields $\la=0$. Therefore, again by \eqref{strfieldlam} we get the desired equality $E_1=E_2$.
\par
Finally, noting that $E_iR_{\perp}\xi\cdot\xi=E_i:\left(\xi\odot R_{\perp}\xi\right)$ for $i=1,2$, and using equality \eqref{strfieldlam} we obtain that
\beq
\ba{ll}
\la^2\left(E_1R_{\perp}\xi\cdot\xi\right)\left(E_2R_{\perp}\xi\cdot\xi\right)
& \dis =\big(E_1:(E_1-E_2)\big)\,\big(E_2:(E_1-E_2)\big)
\\ \ecart
& \dis =\left(|E_1|^2+|E_2|^2\right)E_1:E_2-|E_1|^2|E_1|^2-(E_1:E_2)^2,
\ea
\eeq
which implies the equivalence between \eqref{E1E2} and \eqref{E1E2b}.
The proof of Theorem~\ref{thm.realam} is now complete.
\cqfd
\par\bigskip\noindent
{\bf Acknowledgment:} The author is very grateful to P. G\'erard for stimulating discussions and in particular for Lemma~\ref{lem.G}.
%%%%%%%%%%


\begin{thebibliography}{10}
\bibitem{Ale}{\sc G.~Alessandrini}: ``Critical points of solutions of elliptic equations in two variables", {\em Ann. Scuola Norm. Sup. Pisa Cl. Sci. (4)}, 14 (2)  (1987), 229-256.

\bibitem{AlNe}{\sc G.~Alessandrini \& V.~Nesi}: ``Univalent $\sigma$-harmonic mappings", {\em Arch. Rational Mech. Anal.}, {\bf 158} (2001), 155-171.

\bibitem{Bre}{\sc A.~Bressan}: {\em Hyperbolic Systems of Conservation Laws. The One-Dimensional Cauchy Problem}, Oxford Lecture Series in Math. and its Appl. {\bf 20}, Oxford University Press, 2000, pp.~250.

\bibitem{BrMi}{\sc M.~Briane \& G.W.~Milton}: ``Isotropic realizability of current fields in $\mathbb{R}^3$", to appear in {\em SIAM J. Appl. Dyn. Sys.}

\bibitem{BMT}{\sc M.~Briane, G.W.~Milton \& A.~Treibergs}: ``Which electric fields are realizable in conducting materials?", {\em ESAIM: Math. Model. Numer. Anal.}, {\bf 48} (2) (2014), 307-323.

\bibitem{Eva}{\sc L.C.~Evans}: {\em Partial Differential Equations}, Graduate Studies in Mathematics Vol. {\bf 19}, American Mathematical Society, Providence, RI, 2010, pp.~749.

\bibitem{Ger}{\sc P. G\'erard}: {\em Personal communication}, 2015.

\bibitem{HaWi}{\sc P.~Hartman \& A.~Wintner}: ``On the local behavior of solutions of non-parabolic partial differential equations (I)", {\em Amer. J. Math.}, {\bf 75} (1953), 449-476.

\bibitem{HSD}{\sc M.~Hirsch, S.~Smale \& R.~Devaney}: {\em Differential Equations, Dynamical Systems, and an Introduction to Chaos}, Pure and Applied Mathematics {\bf 60}, Amsterdam, 2004,  pp.~431.

\bibitem{Hor}{\sc L.~H\"ormander}: {\em Lectures on Nonlinear Hyperbolic Differential Equations}, Mathematics \& Applications Vol.~{\bf 26}, Springer-Verlag, Berlin, 1997, pp.~289.

\bibitem{Mil}{\sc G.W.~Milton}: {\em The Theory of Composites}, Cambridge Monographs on Applied and Computational Mathematics, Cambridge University Press, Cambridge 2002, pp.~719.

\bibitem{Sch}{\sc F.~Schulz}: {\em Regularity Theory for Quasilinear Elliptic Systems and Monge-Amp\`ere Equations in Two Dimensions}, Lecture Notes in Mathematics, Springer-Verlag, Berlin, 1990, pp.~123.



\end{thebibliography}
\end{document}